\documentclass[12pt]{article}

\DeclareMathAlphabet\mathbfcal{OMS}{cmsy}{b}{n}

\usepackage{geometry}
\usepackage[dvipsnames]{xcolor}
\usepackage{enumerate}
\usepackage[utf8]{inputenc}
\usepackage{amsfonts}
\usepackage{amsthm}
\usepackage{lscape}
\usepackage{mathrsfs}
\usepackage{graphicx}
\usepackage{tikz-cd}
\usepackage{calc}
\usepackage{float}
\usepackage{multirow}
\usepackage{lipsum}
\usepackage{amssymb}
\usepackage{mathtools}
\usepackage{mathdots}
\usepackage{fancyhdr}
\usepackage{tikz}
\usepackage{tikz-cd}
\usepackage{faktor}
\usepackage{setspace}
\usetikzlibrary{decorations.markings}
\newtheorem{theorem}{Theorem}
\theoremstyle{plain}
\usepackage{sectsty}
\usepackage{kpfonts}
\usepackage{blindtext}
\usepackage[all]{xy}
\usepackage{scalerel}
\newtheorem{corollary}{Corollary}
\newtheorem{definition}{Definition}
\newtheorem{example}{Example}
\newtheorem{lemma}{Lemma}

\newtheorem{proposition}{Proposition}
\newtheorem{remark}{Remark}

\newtheorem*{observation*}{Observation}
\newtheorem*{proposition*}{Proposition}
\newtheorem*{theorem*}{Theorem}
\newtheorem*{claim*}{Claim}

\newcommand{\Aut}{\mbox{Aut}}

\newcommand{\Isom}{\mbox{Isom}}
\newcommand{\rank}{\mbox{rank}}
\newcommand{\cl}{\mbox{cl}}

\title{Real Einstein loci}
\date{}
\author{Gabriella Clemente}
\begin{document}
\maketitle

\begin{abstract}
The aim of this article is to study the interplay between the complex, and underlying real geometries of a K{\"a}hler manifold. We provide a necessary and sufficient condition for certain anti-holomorphic automorphisms of a compact K{\"a}hler-Einstein manifold to determine real Einstein submanifolds.
\end{abstract}
\noindent


\noindent \textbf{Mathematics~Subject~Classification:}
53C25, 32Q20, 32Q26, 14P99.

\section{Introduction}\label{1}

The study of submanifolds of Einstein spaces with special curvature properties is fairly ancient. Consider, for instance, space forms, which are the simplest examples of Einstein manifolds. It is easy to see that totally geodesic, and more generally, totally umbilical submanifolds of a space form are also space forms. As soon as we relax the constant curvature assumption on the ambient space, we run into trouble, e.g.\ totally geodesic submanifolds of an Einstein manifold are not necessarily Einstein. Nevertheless, there are results providing sufficient conditions for submanifolds of Einstein manifolds to be Einstein. It is known that totally umbilical submanifolds of an Einstein manifold with constant mixed sectional curvature are Einstein \cite{AL}. Totally umbilical hypersurfaces in an Einstein manifold have constant scalar curvature, and any manifold of constant scalar curvature can be isometrically embedded as a totally geodesic hypersurface into a uniquely given Einstein manifold \cite{Koiso}.

In the complex setting, there are also results on the geometry of submanifolds subject to various constraints placed on the curvature of the ambient space. In \cite{Yano}, it is shown that totally umbilical totally real submanifolds of a K{\"a}hler manifold with vanishing Bochner curvature tensor are conformally flat. The existence of totally geodesic real hypersurfaces in a K{\"a}hler-Einstein manifold poses severe restrictions on the ambient geometry \cite{Koiso}.

The goal of this article is to study interactions between anti-holomorphic maps, such as real structures, and the metric geometry of K{\"a}hler manifolds. We characterize real Einstein submanifolds of compact K{\"a}hler-Einstein manifolds that are fixed point loci of certain anti-holomorphic maps. Our characterization involves a curvature trace operator. We prove that such a real locus is Einstein iff this trace operator is orthogonally diagonalizable with single real eigenvalue (c.f.\ Theorem \ref{maint}). 

The organization is as follows. In section \ref{2A}, we review Riemannian and complex differential geometry concepts that will be needed later on. Section \ref{2B} is about anti-holomorphic isometries of K{\"a}hler manifolds. In section \ref{2C}, we produce examples of totally geodesic real submanifolds of complex projective manifolds, including smooth projective toric varieties. Section \ref{3} is dedicated to the proof of Theorem \ref{maint}. We conclude with some remarks on the Lagrangian nature of real loci in compact K{\"a}hler-Einstein manifolds.

\section{Anti-holomorphisms and K{\"a}hler metrics}\label{2}

\subsection{Background}\label{2A}
 
Let $(M,g)$ be a Riemannian manifold and $(V, g|_V)$ be a connected submanifold. Recall that $(V, g|_V)$ is totally geodesic if any geodesic in $(V, g|_V)$ stays a geodesic in $(M,g).$ For example, geodesics are always $1$-dimensional totally geodesic submanifolds. 

The second fundamental form of $V$ is the difference $h(\zeta,\eta):=\nabla^M_{\zeta} \eta - \nabla^V_{\zeta} \eta,$ where $\nabla^M$ and $\nabla^V$ are the Levi-Civita connections of $g,$ respectively $g|_V.$ For the proofs of Proposition \ref{tgs} and Theorem \ref{totig}, see pages $95 - 96$ of \cite{Kir}.

\begin{proposition}\label{tgs}
The submanifold $(V, g|_V)$ of $(M,g)$ is totally geodesic iff $h=0.$ 
\end{proposition}

Suppose that $(V, g|_V)$ is totally geodesic in $(M,g).$ Then, by Proposition \ref{tgs}, we have that for any vector fields $\zeta,\eta,\rho,\upsilon \in \mathfrak{X}(V),$ \[Rm^V(\zeta,\eta,\rho,\upsilon)=Rm^M(\zeta,\eta,\rho,\upsilon),\] where $Rm^V$ is the Riemann curvature tensor of $g|_V$ and $Rm^M$ is that of $g$ (c.f.\ Corollary 1.11.2 \cite{Kir}). 

\begin{theorem}\label{totig}
Each connected component of the fixed point set of an isometry of $(M,g)$ is a totally geodesic submanifold.
\end{theorem}

Let $(X,J)$ be a complex manifold. Recall that a smooth map $\sigma:X \rightarrow X$ is holomorphic iff $\sigma_* J = J \sigma_*,$ and it is anti-holomorphic iff $\sigma_* J = -J \sigma_*.$ A bi-anti-holomorphism is an anti-holomorphic map with anti-holomorphic inverse. A bi-anti-holomorphism $\sigma$ satisfies $-J=\sigma^* J := \sigma_*^{-1} J \sigma_*.$ 

An anti-holomorphic involution $\sigma$ of $(X,J)$ is often called a real structure. The real part of $X$ w.r.t.\ $\sigma$ is the fixed point set of $\sigma.$ Note that the fixed point set of a real structure could be empty, however as soon as it is not, we know that it will be a real submanifold of dimension equaling the complex dimension of $X.$ It should be noted that the real part could have more than one connected component. More information on real structures can be found in \cite{REAL}.

Assume that the real dimension of $X$ is $2n.$ An $n$-dimensional submanifold $Y$ of $X$ is said to be totally real if ${T_X}|_Y \simeq T_Y \oplus J(T_Y).$

Let $(X,g,J)$ be a K{\"a}hler manifold. The associated K{\"a}hler form is defined as $\omega(\cdot,\cdot)=g(J\cdot,\cdot).$ So, $g(\cdot,\cdot)=\omega(\cdot,J\cdot).$ In local holomorphic coordinates, $\omega=ig_{j\bar{k}}dz^j \wedge d\bar{z}^k,$ and $g=g_{j \bar{k}}(dz^j \otimes d \bar{z}^k + d \bar{z}^k \otimes dz^j).$

The Ricci form of $\omega$ is given as \[Ric(\omega)(\zeta,\eta):=Ric_g(J\zeta,\eta),\] for any $\zeta,\eta \in \mathfrak{X}(X).$ Certainly, $g$ is K{\"a}hler-Einstein iff $\omega$ is, i.e.\ $Ric_g =\lambda g$ iff $Ric(\omega)=\lambda \omega.$

\subsection{Properties of anti-holomorphic maps}\label{2B}
For the terminology and complex geometry concepts that we make use of here, see \cite{cg}.

\begin{lemma}\label{L1}
Let $X$ be an $n$-dimensional complex manifold and $f:X \rightarrow X$ be an anti-holomorphic map. Then $\partial f^*=f^* \bar{\partial}$ and $\bar{\partial} f^* = f^* \partial.$
\end{lemma}

\begin{proof}
Let $p \in X,$ $(U_i, \phi_i)$ be a chart containing $p,$ $(U_j,\phi_j)$ be a chart containing $f(p),$ and suppose that $f(U_i) \subseteq U_j.$ Write $f_k$ for the $k^{th}$ component of $\phi_j \circ f \circ \phi_i^{-1},$ which is a function of $(z^1,\dots,z^n) \in \phi_i(U_i) \subset \mathbb{C}^n.$ Interpreting partial derivatives in the adequate sense, we have that

\[f^*(dz^k)=\frac{\partial f_k}{\partial z^m} dz^m + \frac{\partial f_k}{\partial \bar{z}^m} d\bar{z}^m=  \frac{\partial f_k}{\partial \bar{z}^m} d\bar{z}^m,\]

since $f_k,$ as the composition of the holomorphic coordinate function $z^k$ and the anti-holomorphic function $f \circ \phi_i^{-1},$ is anti-holomorphic. 
Similarly, 

\[f^*(d\bar{z}^k)=\frac{\partial \bar{f}_k}{\partial z^m} dz^m + \frac{\partial \bar{f}_k}{\partial \bar{z}^m} d\bar{z}^m=  \frac{\partial \bar{f}_k}{\partial z^m} dz^m,\]

since $f_k$ is anti-holomorphic and hence $\bar{f}_k$ is holomorphic. These calculations imply that the pull-back $f^*$ induces type-reversing maps $\Omega^{p,q}(X) \rightarrow \Omega^{q,p} (X).$ 

Let $\Pi^{p,q}: \Omega^{k}(X)=\bigoplus_{r+s=k} \Omega^{r,s}(X) \rightarrow \Omega^{p,q}(X)$ be the projection operator so that $\partial= \Pi^{p+1,q} \circ d$ and $\bar{\partial}= \Pi^{p,q+1} \circ d,$ where $d:\Omega^{p,q}(X) \rightarrow \Omega^{p+1,q} (X) \oplus \Omega^{p,q+1} (X)$ is the $\mathbb{C}-$linear extension of the exterior derivative. Let $\beta=\sum_{r+s=k} b_{i_r,j_s} dz^{i_1} \wedge \dots \wedge dz^{i_r} \wedge d \bar{z}^{j_1} \wedge \dots  \wedge d \bar{z}^{j_s} \in \Omega^k(X).$ Then, 

\begin{equation*}
\begin{split}
f^* \big(\Pi^{p,q} \beta \big) & = f^*(b_{i_p,j_q} dz^{i_1} \wedge \dots \wedge dz^{i_p} \wedge d \bar{z}^{j_1} \wedge \dots  \wedge d \bar{z}^{j_q}) \\
&= (b_{i_p,j_q} \circ f) (f^*dz^{i_1}) \wedge \dots  (f^*dz^{i_p}) \wedge (f^* d \bar{z}^{j_1}) \wedge \dots  \wedge (f^*d \bar{z}^{j_q}) \\
&= \Pi^{q,p} \Big( \sum_{r+s=k} (b_{i_r,j_s} \circ f) (f^*dz^{i_1}) \wedge \dots  (f^*dz^{i_r}) \wedge (f^* d \bar{z}^{j_1}) \wedge \dots  \wedge (f^*d \bar{z}^{j_s}) \Big) \\
& = \Pi^{q,p} (f^* \beta).
\end{split}
\end{equation*}

Therefore, for any $ \gamma \in \Omega^{p,q} (X),$

\begin{equation*}
\begin{split}
\partial (f^* \gamma) & = \big( \Pi^{q+1,p} \circ d \big) (f^* \gamma) \\
&=\big( \Pi^{q+1,p} \circ f^* \big)d \gamma \\
&=f^* \big(\Pi^{p,q+1} \circ d) \gamma \\
&=f^*(\bar{\partial} \gamma). 
\end{split}
\end{equation*}

A similar computation shows that also $\bar{\partial} (f^* \gamma ) = f^* (\partial \gamma).$ 
\end{proof}

\begin{lemma}\label{spg}
Let $(X,g,J)$ be a K{\"a}hler manifold with K{\"a}hler form $\omega.$ An anti-holomorphic anti-isometry of $(X,\omega)$ is an isometry of $(X,g).$
\end{lemma}

\begin{proof}
Suppose that $f^* \omega=-\omega.$ For any vector fields $\zeta,\eta \in \mathfrak{X}(X),$ we have that 

\begin{equation*}
\begin{split}
(f^*g)(\zeta,\eta)&=g(f_* \zeta,f_* \eta)\\
&=\omega(f_* \zeta, Jf_* \eta)\\
&=-\omega(f_* \zeta, f_* J\eta)\\
&=-(f^* \omega)(\zeta,J\eta)\\
&=\omega(\zeta,J\eta)\\
&=g(J\zeta,J\eta)\\
&=g(\zeta,\eta),
\end{split}
\end{equation*}
so $f$ is an isometry of $(X,g).$
\end{proof}

\begin{lemma}\label{L2}
Let $(X,g,J)$ be a K{\"a}hler manifold with K{\"a}hler form $\omega,$ and $Y$ be an open complex submanifold. Suppose that the induced K{\"a}hler form $\omega|_Y$ is given in terms of a K{\"a}hler potential $\psi,$ and that there is an anti-holomorphic map $f:X \rightarrow X$ that preserves $Y$ and that satisfies $f^* \psi = \psi.$ Then, $f$ is an anti-isometry of $(Y, \omega|_Y),$ hence an isometry of $(Y, g|_Y).$ 

\begin{proof}
By Lemma \ref{L1}, it follows that

\begin{equation*}
\begin{split}
f^* \omega & = i f^* (\partial \bar{\partial} \psi) \\
&=i \bar{\partial} \partial f^* \psi \\
&=-i\partial \bar{\partial}\psi \\
&=-\omega.
\end{split}
\end{equation*}

Lemma \ref{spg} implies that $f$ is an isometry of $(Y, g|_Y).$
\end{proof}
\end{lemma}

\subsection{Examples of totally geodesic real parts}\label{2C}

\begin{example}\label{exi1}
Consider the Fubini-Study form $\omega_{FS}$ on $\mathbb{CP}^n.$ Recall that to any point $[z_0:\dots:z_n]$ in the open subset $U_i=\{[z_0:\dots:z_n] \in \mathbb{CP}^n | z_i \neq 0\}$ corresponds the coordinate $\big(\frac{z_0}{z_i},\dots,\frac{z_{i-1}}{z_i},\frac{z_{i+1}}{z_i},\dots,\frac{z_n}{z_i}\big) \in \mathbb{C}^n.$ For each $j\neq i,$ put $w_j=\frac{z_j}{z_i}$ and $\|w\|^2=\sum_{j\neq i} |w_j|^2.$ Then, $\omega_{FS}=i \partial \bar{\partial} \log{(1 + \|w\|^2)}.$ The associated Riemannian metric is \[g=\Big(\frac{(1 + \|w\|^2)\delta_a^{\bar{b}} - \bar{w}_a w_b}{(1 + \|w\|^2)^2}\Big)(dw_a \otimes d \bar{w}_b + d \bar{w}_b \otimes dw_a).\]

Define a complex conjugation $\sigma:\mathbb{CP}^n \rightarrow \mathbb{CP}^n$ by $\sigma([z_0:\dots:z_n])=[\bar{z}_0:\dots:\bar{z}_n].$ The map $\sigma$ is anti-holomorphic, and in fact, it is a real structure. Observe that \[\sigma^* \log{(1 + \|w\|^2)}= \log{\big(1 + \sum_{j\neq i} |\bar{w}_j|^2\big)}= \log{(1 + \|w\|^2)}.\] So by Lemma \ref{L2}, $\sigma^* g=g$ on each K{\"a}hler submanifold $U_i \subset \mathbb{CP}^n,$ and this is precisely what it means for $\sigma$ to be an isometry of $(\mathbb{CP}^n,g).$ 

The real part of $\mathbb{CP}^n$ w.r.t.\ $\sigma$ is $\{[z_0:\dots:z_n] \in \mathbb{CP}^n | \Im(z_i) = 0 \mbox{ for all }i\}=\mathbb{RP}^n.$ By Theorem \ref{totig}, it follows that $(\mathbb{RP}^n, g|_{\mathbb{RP}^n})$ is totally geodesic. 
\end{example}

It is worth mentioning that $\mathbb{RP}^n$ is one of the simplest examples of a real toric variety, as defined in \cite{SF}. Moreover, $\mathbb{RP}^n$ can be turned into a small cover of the standard $n$-simplex by taking the real torus action to be the induced subgroup action of $(\mathbb{Z}_2)^n \subseteq (\mathbb{R}^*)^n \subseteq (\mathbb{C}^*)^n.$ Small covers were introduced in \cite{DJ}.

\begin{example}\label{nokia}
Let $f_1,\dots,f_k \in \mathbb{C}[z_0,\dots,z_n]$ be homogeneous polynomials with real coefficients, i.e.\ $f_i=\sum_{j=1}^n q_j z_0^{\alpha_0^j}\dots z_n^{\alpha_n^j},$ where $\sum_{k=0}^n \alpha_k^j$ is a constant and $q_j \in \mathbb{R}.$ If $V:=\mathbb{V}(f_1,\dots,f_k) \subseteq \mathbb{CP}^n$ is smooth, then $(V, \omega_{FS}|_V)$ is a K{\"a}hler manifold, and the restricted complex conjugation map $\sigma|_V$ is an isometry of $(V, g|_V)$ (c.f.\ Example \ref{exi1}). Note that the fixed point set of $\sigma|_V$ is the real part $V \cap \mathbb{RP}^n.$ So by Theorem \ref{totig}, $V \cap \mathbb{RP}^n$ with the induced metric is a totally geodesic submanifold of $(V,g|_V).$ 

Here is a specific instance of this construction: take $f=\sum_{j=1}^{n+1} z_j^2 - z_0^2,$ and consider the variety $V:=\mathbb{V}(f) \subset \mathbb{CP}^{n+1}.$ The fixed point set of $\sigma|_V$ is \[V \cap \mathbb{RP}^{n+1} \simeq S^n,\] which must then be a totally geodesic submanifold of $(V,g|_V).$ 
\end{example}

We conclude this section with an observation about toric varieties, and their totally geodesic subvarieties, which is Proposition \ref{ttoorrii}. The proof of this proposition makes use of the following topological remark.

\begin{lemma}\label{auxl} 
Let $T$ and $T'$ be topological spaces and assume that $T'$ is Hausforff. Let $f_1, f_2:T \rightarrow T'$ be continuous maps and $U \subseteq T$ be a dense subset. If ${f_1}|_U = {f_2}|_U,$ then $f_1=f_2$ on $T.$
\end{lemma}

\begin{proof}
Let $C=\{x \in T | f_1(x)=f_2(x)\} \subseteq T.$ Define a function $F:T \rightarrow T' \times T'$ by $F(x)=(f_1(x),f_2(x)),$ which is continuous. Since products and subspaces of Hausdorff spaces are Hausdorff, $T' \times T'$ is Hausdorff, and so is $F(T) \subseteq T' \times T'.$ Hence, the diagonal $\Delta=\{(a,b)| a=b\} \subset F(T)$ is closed. But since $C=F^{-1}(\Delta),$ $C$ is a closed subspace of $T.$ By definition, the closure $\cl(U)$ is the smallest closed subset of $T$ that contains $U,$ and clearly $U \subseteq C.$ Therefore, $T=cl(U) \subseteq C.$ But then $T=C,$ which means that $f_1=f_2$ on $T.$
\end{proof}

\begin{proposition}\label{ttoorrii}
Let $X$ be a smooth projective toric variety of dimension $n,$ equipped with a torus invariant K{\"a}hler form $\omega.$ Let $g$ denote the Riemannian metric defined by $\omega$ and the given complex structure. The real part of $X$ with respect to a real structure extending coordinate-wise complex conjugation on $(\mathbb{C}^*)^n$ is a totally geodesic submanifold of $(X, g).$ 
\end{proposition}

\begin{proof}
Let $T_{\mathbb{C}}^n \subseteq X$ denote the open dense $(\mathbb{C}^*)^n$-orbit. Then, $\omega|_{T_{\mathbb{C}}^n}=i \partial \bar{\partial} f,$ for a convex real-valued function $f$ of the real variables $x_i=\log{|z_i|^2},$ where the $z_i$ are the standard holomorphic coordinates on $(\mathbb{C}^*)^n.$ 

Consider an anti-holomorphic involution $\sigma:X \rightarrow X$ such that the restriction $\sigma|_{T_{\mathbb{C}}^n}$ is complex conjugation. Then, $\sigma|_{T_{\mathbb{C}}^n}^* f =f,$ and so by Lemma \ref{L2}, $\sigma|_{T_{\mathbb{C}}^n}$ is an anti-isometry of $(T_{\mathbb{C}}^n, \omega|_{T_{\mathbb{C}}^n})$ and an isometry of $(T_{\mathbb{C}}^n, g|_{T_{\mathbb{C}}^n}).$ 

We can think of both $\sigma^* g$ and $g$ as continuous functions from the toric manifold $X$ into the space $B$ of all symmetric bilinear forms on each tangent space of $X,$ which is a manifold, and hence is Hausdorff. Applying Lemma \ref{auxl} to $f_1=\sigma^* g,$ $f_2=g,$ $T=X,$ $T'=B,$ and $U=T_{\mathbb{C}}^n,$ which recall is open and dense in $X,$ we conclude that $\sigma$ is an isometry of $(X,g).$ The claim now follows from Theorem \ref{totig}. 
\end{proof}

\section{Submanifolds of K{\"a}hler-Einstein manifolds}\label{3}

In this section, we prove the main theorem, which is Theorem \ref{maint} (see also Definition \ref{maindef}). The proof involves several steps. The main idea is to exploit uniqueness of K{\"a}hler-Einstein metrics, and the behavior of bi-anti-holomorphisms. We are interested in those cases where bi-anti-holomorphic mappings give rise to totally real and totally geodesic Einstein submanifolds. Note that total realness or being totally geodesic (or both) is not enough to deduce the Einstein condition from an ambient K{\"a}hler-Einstein condition.

Let $X$ be a compact K{\"a}hler manifold. Suppose that $X$ admits a K{\"a}hler-Einstein (KE) metric $\omega.$ Thus, $\omega$ satisfies the equation $Ric(\omega)=\lambda \omega,$ for some real number $\lambda.$ After rescaling $\omega$ by the factor $|\lambda|^{-1},$ we may assume that the Einstein constant is either $\pm 1$ or $0.$ Recall that $c_1(X)=\frac{1}{2\pi} [Ric(\omega)].$ So if $Ric(\omega)=\omega,$ then $c_1(X)>0;$ if $Ric(\omega)=0,$ then $c_1(X)=0;$ and if $Ric(\omega)=-\omega,$ then $c_1(X)<0.$ Assume that $c_1(X)>0,$ and let $\omega'$ be another KE form on $X.$ Since the $1^{st}$ Chern class of $X$ is independent of $\omega,$ $[\omega']=2\pi c_1(X)=[\omega],$ so $\omega' \in 2\pi c_1(X).$ In brief, when $c_1 (X)>0,$ any KE form on $X$ must belong to $2\pi c_1(X).$ Similarly, if $c_1(X)<0,$ then any KE form on $X$ must belong to $-2\pi c_1(X).$

Let $\Aut_0(X)$ be the identity component of the group of holomorphic automorphisms of a Fano KE manifold $X.$ This group acts on the set of K{\"a}hler metrics in the class $2\pi c_1(X),$ where the action is $(\omega,f) \mapsto f^* \omega.$ In fact, the subset of KE metrics is a single $\Aut_0(X)$-orbit \cite{BM}. 

\begin{theorem}\label{uniq}
Let $X$ be a compact K{\"a}hler manifold. 

\begin{enumerate}
\item If $c_1(X)<0,$ then $X$ carries a unique K{\"a}hler-Einstein metric $\omega \in -2\pi c_1(X),$ where $Ric(\omega)=-\omega$ \cite{Aubin, Yau}.

\item If $c_1(X)=0,$ then each K{\"a}hler class of $X$ contains a unique K{\"a}hler-Einstein metric, where $Ric(\omega)=0$ \cite{Yau}.

\item If $c_1(X)>0,$ and if a K{\"a}hler-Einstein metric $\omega \in 2\pi c_1(X)$ exists, in which case, $Ric(\omega)=\omega,$ then it is unique up to the action of $\Aut_0(X),$ i.e.\ if $\omega' \in 2\pi c_1(X)$ is another K{\"a}hler-Einstein form, then there must exist an $a \in \Aut_0(X)$ such that $\omega'=a^* \omega$ \cite{BM}.
\end{enumerate}
\end{theorem}

If $c_1(X)>0,$ then a KE metric does not always exist. A Fano manifold is KE iff its anti-canonical polarization satisfies an algebro-geometrical condition known as K-stability \cite{CDS, Tian}. On that note, it could be interesting to develop a notion of K-stability over $\mathbb{R}$ with the purpose of obtaining a new necessary and sufficient condition for the existence of an Einstein metric on a compact real manifold. The results of the present work can serve as a testing ground for these ideas. 

Next, we prove a number of preliminary results that will go into the proof of Theorem \ref{maint}. When needed, we work with bi-anti-holomorphic maps instead of anti-holomorphic ones. This is to ensure that the pullback of the complex structure in question is defined, and also to prevent the pullback of K{\"a}hler forms being degenerate.

\begin{lemma}\label{totrk}
Let $n\geq 1,$ $(X,J)$ be a complex manifold of real dimension $2n,$ $f:(X,J) \to (X,J)$ be an anti-holomorphic mapping, and denote its fixed point set by $X^f.$ If $X^f$ is an $n$-dimensional submanifold of $X,$ then $X^f$ is totally real. In particular, \[f_*|_{T_{X^f}}=\operatorname{Id}_{T_{X^f}} \mbox{ and } f_*|_{J(T_{X^f})}=-\operatorname{Id}_{J(T_{X^f})}.\]
\end{lemma}

\begin{proof}
Assume that $X^f$ is a submanifold of real dimension $n.$ By the definition of fixed point set, \[T_{X^f}=\{\zeta \in T_X \mid f_* \zeta=\zeta\}.\] Since $f$ is anti-holomorphic, for any $\zeta' \in T_{X^f},$ $f_* J\zeta'=-J\zeta'.$ Thus,

\[J(T_{X^f})=\{\zeta \in T_X \mid f_* \zeta=-\zeta\}.\] Indeed, the above observations confirm that $f_*|_{T_{X^f}}=\operatorname{Id}_{T_{X^f}},$ and $f_*|_{J(T_{X^f})}=-\operatorname{Id}_{J(T_{X^f})}.$ It is clear that these subbundles of $T_X$ have trivial intersection, i.e.\ $T_{X^f} \cap J(T_{X^f})=\{0\}.$ Now, note that \[\rank_{\mathbb{R}} \big({T_{X^f} \oplus J(T_{X^f})}\big)=2 \rank_{\mathbb{R}} \big({T_{X^f}}\big)=2n.\] Therefore, ${T_X}|_{X^f} \simeq T_{X^f} \oplus J(T_{X^f}),$ and so $X^f$ is a totally real submanifold of the complex manifold $(X,J).$ 
\end{proof}

\begin{lemma}\label{Kd}
Let $(X,g,J)$ be a K{\"a}hler manifold, and $\nabla$ be the Levi-Civita connection. The Riemann curvature endomorphism satisfies 

\[R^{\nabla}(\zeta,\eta)(J\rho)=J R^{\nabla} (\zeta,\eta)\rho,\] for any vector fields $\zeta,\eta,\rho \in \mathfrak{X}(X).$
\end{lemma}

\begin{proof}
Since $J$ is K{\"a}hler, $(\nabla_{\zeta} J)(J\eta)=0,$ and so $\nabla_{\zeta} J\eta = J(\nabla_{\zeta} \eta).$ Thus, 

\begin{equation*}
\begin{split}
R^{\nabla}(\zeta,\eta)(J\rho)&=\nabla_{\zeta} \nabla_{\eta} J\rho-\nabla_{\eta} \nabla_{\zeta} J\rho -\nabla_{[\zeta,\eta]} J\rho\\
&=\nabla_{\zeta} \big(J(\nabla_{\eta} \rho)\big)-\nabla_{\eta} \big(J(\nabla_{\zeta} \rho)\big)-J\big(\nabla_{[\zeta,\eta]} \rho\big)\\
&=J \big(\nabla_{\zeta} \nabla_{\eta} \rho-\nabla_{\eta} \nabla_{\zeta} \rho -\nabla_{[\zeta,\eta]} \rho \big)\\
&= J R^{\nabla} (\zeta,\eta)\rho.
\end{split}
\end{equation*}
\end{proof}

\begin{lemma}\label{chi}
Let $(X,J)$ be a complex manifold and $f:(X,J) \rightarrow (X,J)$ be a bi-anti-holomorphic map. If $\omega \in 2 \pi c_1(X),$ then $-f^* \omega \in 2 \pi c_1(X).$
\end{lemma}

\begin{proof}
Since we are assuming that our choice of manifold $X$ is fixed, the first Chern class of $X$ depends on the complex structure only. Thus, we write $c_1(J)$ in place of $c_1(X):=c_1(T_X).$ 
Now, simply note that 

\begin{equation*}
\begin{split}
[-f^* \omega] & = -f^* [\omega] \\
&=-f^* (2 \pi c_1(J))\\
&=-2 \pi c_1(f^*J) \\
&=-2 \pi c_1(-J) \\
&=[\omega].
\end{split}
\end{equation*}
Here we have used the naturality of Chern classes to go from the second to the third equality, and the identity $c_k(-J)=(-1)^k c_1(J)$ (see Lemma 14.9 in \cite{c1}) to deduce the last line.
\end{proof}

\begin{lemma}\label{zch}
Let $(X,J)$ be a complex manifold, $\omega$ be a K{\"a}hler-Einstein form on $X,$ and $f:(X,J) \rightarrow (X,J)$ be a bi-anti-holomorphic map. Then $-f^* \omega$ is a K{\"a}hler-Einstein form on $X.$
\end{lemma}

\begin{proof}
Let us write $g$ for the Riemannian metric associated to $\omega.$ Observe that 

\begin{equation*}
\begin{split}
f^* g(Y,Z) & = g(f_{*} Y, f_{*} Z) \\
&=\omega(f_{*} Y,  J f_{*}Z)\\
&=-\omega(f_{*} Y,  f_{*} J Z)\\\
&=-f^* \omega(Y, JZ).
\end{split}
\end{equation*}

Since locally we have that $\omega=i g_{i \bar{j}} dz^i \wedge d \bar{z}^j,$ the anti-holomorphicity of $f$ implies that $f^* \omega = i (f^* g_{i \bar{j}}) d \bar{z}^j \wedge dz^i = i (-f^* g_{i \bar{j}}) dz^i \wedge d \bar{z}^j.$ Using Lemma \ref{L1}, we compute

\[f^* Ric(\omega) = f^*(-i \partial \bar{\partial} \log{\det(g)}) =-i \partial \bar{\partial} \log{\det(-f^*g)} = Ric(f^* \omega).\]

Assuming that $\lambda$ is the Einstein constant of $\omega,$ it follows that 

\begin{equation*}
\begin{split}
Ric(-f^* \omega)&=-f^* Ric(\omega)\\
&=\lambda (-f^* \omega).
\end{split}
\end{equation*}

\end{proof}

\begin{lemma}\label{ric}
Let $n\geq 1,$ and $(X,g,J)$ be a K{\"a}hler manifold with $\dim_{\mathbb{R}} {X}=2n.$ Let $(Y,g|_Y)$ be a totally real totally geodesic submanifold, and $(e_1,\dots,e_n,Je_1,\dots,Je_n)$ be a local orthonormal frame of ${T_X}|_Y$ such that $(e_1,\dots,e_n)$ is a local orthonormal frame of $T_Y.$ For any $\zeta,\eta \in \mathfrak{X}(Y),$

\[Ric_{g|_Y} (\zeta,\eta)=Ric_g(\zeta,\eta)-\sum_{1\leq \alpha \leq n} Rm^X(Je_{\alpha},\zeta,\eta,Je_{\alpha}).\]
\end{lemma}

\begin{proof}
On the one hand, we have that since $(Y,g|_Y)$ is totally geodsic \[Ric_{g|_Y} (\zeta,\eta):=\sum_{1\leq \alpha \leq n} Rm^Y(e_{\alpha},\zeta,\eta,e_{\alpha})=\sum_{1\leq \alpha \leq n} Rm^X(e_{\alpha},\zeta,\eta,e_{\alpha}).\]

On the other hand, \[Ric_g(\zeta,\eta)=\sum_{1\leq \alpha \leq n} [Rm^X(e_{\alpha},\zeta,\eta,e_{\alpha})+Rm^X(Je_{\alpha},\zeta,\eta,Je_{\alpha})].\] Thus, \[Ric_{g|_Y} (\zeta,\eta)=Ric_g(\zeta,\eta)-\sum_{1\leq \alpha \leq n} Rm^X(Je_{\alpha},\zeta,\eta,Je_{\alpha}).\]
\end{proof}

\begin{lemma}\label{iffKE}
Let $n\geq 1,$ $(X,g,J)$ be a K{\"a}hler-Einstein manifold with $\dim_{\mathbb{R}} {X}=2n,$ $(Y,g|_Y)$ be a totally real totally geodesic submanifold, and $(e_1,\dots,e_n,Je_1,\dots,Je_n)$ be a frame as in Lemma \ref{ric}. Then, $g|_Y$ is Einstein if and only if there is a constant $C$ such that for all $\zeta,\eta \in \mathfrak{X}(Y),$ \[\sum_{1\leq \alpha \leq n} Rm^X(Je_{\alpha},\zeta,\eta,Je_{\alpha})=Cg(\zeta,\eta).\] 
\end{lemma}

\begin{proof}
Assume that $Ric_g=\lambda g.$ By Lemma \ref{ric}, $Ric_{g|_Y}=\kappa g|_Y,$ for some $\kappa \in \mathbb{R},$ if and only if for all $\zeta,\eta \in \mathfrak{X}(Y),$ \[\sum_{1\leq \alpha \leq n} Rm^X(Je_{\alpha},\zeta,\eta,Je_{\alpha})=Cg(\zeta,\eta),\] where $C=\lambda-\kappa.$ 
\end{proof}

The notion of mixed sectional curvature was introduced in \cite{AL} (see the definition on page $662$ \cite{AL}). In the setting of Lemma \ref{iffKE}, assume that the mixed sectional curvature of $(X,g,J)$ with respect to $Y$ is constant and equal to $c,$ i.e.\ $Rm^X (\zeta,J\eta,J\eta,\zeta)=c,$ for any unit vector fields $\zeta,\eta \in \mathfrak{X}(Y).$ From the proof of Theorem 3.2 of \cite{AL}, it follows that $Rm^X (Je_{\alpha},\zeta,\eta,Je_{\alpha})=cg(\zeta,\eta),$ for all $\zeta,\eta \in \mathfrak{X}(Y).$ Thus, $g|_Y$ is Einstein because \[\sum_{1 \leq \alpha \leq n} Rm^X (Je_{\alpha},\zeta,\eta,Je_{\alpha})=nc g(\zeta,\eta).\]

\begin{lemma}\label{TN}
Let $n\geq 1,$ $(X,J)$ be a complex manifold of real dimension $2n,$ $f:(X,J) \to (X,J)$ be an anti-holomorphic mapping, and $X^f$ be its fixed point set. Assume that $X^f$ is an $n$-dimensional submanifold of $X.$ Let $\zeta \in \mathfrak{X}(X).$ Denote by $\zeta^{\top}$ and $\zeta^{\perp}$ the tangent, respectively normal parts of $\zeta.$ We have the following facts : 

\begin{enumerate}
\item $\zeta^{\top}=\frac{f_* \zeta + \zeta}{2}$ and $\zeta^{\perp}=\frac{\zeta-f_* \zeta}{2},$ and 

\item $J\zeta^{\top}=(J\zeta)^{\perp}$ and $J\zeta^{\perp}=(J\zeta)^{\top}.$
\end{enumerate}
\end{lemma}

\begin{proof}
Since $X^f$ is totally real, we can write $\zeta=\zeta_1+J\zeta_2,$ for some uniquely given $\zeta_1,\zeta_2 \in \mathfrak{X}(X^f).$ Indeed, $\zeta^{\top}=\zeta_1$ and $\zeta^{\perp}=J\zeta_2.$ The first point follows from Lemma \ref{totrk} and the anti-holomorphic character of $f.$ The second point follows from the first one and anti-holomorphicity.
\end{proof}

\begin{proposition}\label{mainp}
Let $n\geq 1,$ $(X,g,J)$ be a K{\"a}hler-Einstein manifold with $\dim_{\mathbb{R}} {X}=2n,$ $f:(X,J) \to (X,J)$ be an anti-holomorphic mapping, and $X^f$ be its fixed point set. Assume that $X^f$ is an $n$-dimensional submanifold of $X.$ Let $(e_1,\dots,e_n,Je_1,\dots,Je_n)$ be a local orthonormal frame of ${T_X}|_{X^f}$ such that $(e_1,\dots,e_n)$ is such a frame of $T_{X^f}.$ Let $\zeta, \eta \in \mathfrak{X}(X^f).$ Then, \[\sum_{1\leq \alpha \leq n} Rm^X(Je_{\alpha},\zeta,\eta,Je_{\alpha})=Cg(\zeta,\eta)\] for some $C\in \mathbb{R}$ if and only if \[\sum_{1\leq \alpha \leq n} J\big[R^{\nabla}(Je_{\alpha},\zeta)e_{\alpha}\big]^{\perp}=-C\zeta.\]
\end{proposition}

\begin{proof}
Observe that 

\begin{equation*}
\begin{split}
\sum_{1\leq \alpha \leq n} Rm^X(Je_{\alpha},\zeta,\eta,Je_{\alpha})&=\sum_{1\leq \alpha \leq n} -Rm^X(Je_{\alpha},\zeta,Je_{\alpha},\eta)\\
&=g\big(\sum_{1\leq \alpha \leq n}-\big[R^{\nabla}(Je_{\alpha},\zeta)(Je_{\alpha})\big]^{\top},\eta).
\end{split}
\end{equation*}

Thus, \[\sum_{1\leq \alpha \leq n} Rm^X(Je_{\alpha},\zeta,\eta,Je_{\alpha})=Cg(\zeta,\eta)\] iff 

\begin{equation}\label{ceq}
\begin{split}
\sum_{1\leq \alpha \leq n}\big[R^{\nabla}(Je_{\alpha},\zeta)(Je_{\alpha})\big]^{\top} +C\zeta=0.
\end{split}
\end{equation}

By Lemmas \ref{Kd} and \ref{TN}.2, \[\big[R^{\nabla}(Je_{\alpha},\zeta)(Je_{\alpha})\big]^{\top}=J\big[R^{\nabla}(Je_{\alpha},\zeta)e_{\alpha}\big]^{\perp}.\] Hence, equation \ref{ceq} can be expressed as \[\sum_{1\leq \alpha \leq n} J\big[R^{\nabla}(Je_{\alpha},\zeta)e_{\alpha}\big]^{\perp}=-C\zeta.\]
\end{proof}

\begin{definition}\label{maindef}
Let $(M,g)$ be a Riemannian manifold with Levi-Civita connection $\nabla,$ and curvature endomorphism $R^{\nabla}.$ Assume that $M$ carries a complex structure $J.$ Let $(N,g|_N)$ be a totally real $n$-dimensional Riemannian submanifold, and $(e_{\alpha})_{\alpha=1}^n$ be a local orthonormal frame of $T_N.$ Consider the map $\mathcal{R}^{g,J}_N : \mathfrak{X}(N)^3 \to \mathfrak{X}(N),$ given by \[\mathcal{R}^{g,J}_N (\zeta,\eta,\rho):=J\big[R^{\nabla}(J\zeta,\eta)\rho\big]^{\perp},\] for any $\zeta,\eta,\rho \in \mathfrak{X}(N).$ Define the trace $\operatorname{tr}_{13}(\mathcal{R}^{g,J}_N) : \mathfrak{X}(N) \to \mathfrak{X}(N)$ by \[\operatorname{tr}_{13}(\mathcal{R}^{g,J}_N)=\sum_{1\leq \alpha \leq n} J\big[R^{\nabla}(Je_{\alpha},\cdot)e_{\alpha}\big]^{\perp}.\]
\end{definition}

Of course, we have that \[\operatorname{tr}_{13}(\mathcal{R}^{g,J}_{X^f})=-C\operatorname{Id}_{\mathfrak{X}(N)}\] iff \[\operatorname{tr}_{13}(\mathcal{R}^{g,J}_N)\zeta=\sum_{1\leq \alpha \leq n} J\big[R^{\nabla}(Je_{\alpha},\zeta)e_{\alpha}\big]^{\perp}=-C\zeta,\] for all $\zeta \in \mathfrak{X}(N).$

\begin{theorem}\label{maint}
Let $n\geq 1,$ $(X,g,J)$ be a compact K{\"a}hler-Einstein manifold of real dimension $2n,$ and $\omega$ be the associated K{\"a}hler form. Let $f$ be a bi-anti-holomorphism of $(X,J),$ and denote its set of fixed points by $X^f.$ Assume that $X^f$ is an $n$-dimensional submanifold. If $c_1(X) < 0,$ then $(X^f,g|_{X^f})$ is Einstein if and only if there is a constant $C\in \mathbb{R}$ such that 
\begin{equation}\label{Einscondition}
\begin{split}
\operatorname{tr}_{13}(\mathcal{R}^{g,J}_{X^f})=-C\operatorname{Id}_{\mathfrak{X}(X^f)}.
\end{split}
\end{equation}
If $c_1(X)=0$ and $[-f^* \omega]=[\omega],$ then $(X^f,g|_{X^f})$ is Einstein if and only if condition \ref{Einscondition} holds true. If $c_1(X)>0$ and $f$ is involutive, there exists a holomorphism $a \in Aut_0(X)$ such that $a\circ f$ is an isometry, and for any such $a,$ if the fixed point set $X^{a\circ f}$ is an $n$-dimensional submanifold, then $(X^{a\circ f},g|_{X^{a\circ f}})$ is Einstein iff condition \ref{Einscondition} holds true. 
\end{theorem}

\begin{proof}
We analyze the cases $c_1(X)<0,$ $c_1(X)=0,$ and $c_1(X)>0$ separately. 

\emph{Case $c_1(X)<0$}: The KE form $\omega$ on $X$ is unique (c.f.\ Theorem \ref{uniq}.1). So Lemma \ref{zch} implies that $-f^* \omega =\omega.$ By Lemma \ref{spg}, $f$ is an isometry of $(X,g).$ By Theorem \ref{totig} and Lemma \ref{totrk}, $(X^f,g|_{X^f})$ is a totally geodesic totally real submanifold of $(X,g).$ Let $(e_1,\dots,e_n,J_1,\dots,Je_n)$ be a local orthonormal frame of ${T_X}|_{X^f}$ s.t.\ $(e_1,\dots,e_n)$ is such a frame of $T_{X^f}.$ By Lemma \ref{iffKE} and Proposition \ref{mainp}, $g|_{X^f}$ is Einstein iff \[\sum_{1\leq \alpha \leq n} J\big[R^{\nabla}(Je_{\alpha},\zeta)e_{\alpha}\big]^{\perp}=-C\zeta,\] for all $\zeta$ and some real constant $C,$ i.e.\ iff \[\operatorname{tr}_{13}(\mathcal{R}^{g,J}_{X^f})=-C\operatorname{Id}_{\mathfrak{X}(X^f)}.\]

\emph{Case $c_1(X)=0$}: By Lemma \ref{zch}, $-f^*\omega$ is another KE form on $X.$ Assuming that $\omega$ and $-f^* \omega$ are in the same cohomology class, Theorem \ref{uniq}.2 implies that $-f^* \omega=\omega.$ In this case too, we directly conclude that $f$ is an isometry of $(X,g),$ and so $(X^f,g|_{X^f})$ is a totoally geodesic totally real submanifold. By Lemma \ref{iffKE} and Proposition \ref{mainp}, $g|_{X^f}$ is Einstein if and only if $\operatorname{tr}_{13}(\mathcal{R}^{g,J}_{X^f})=-C\operatorname{Id}_{\mathfrak{X}(X^f)},$ for some $C\in \mathbb{R}.$

\emph{Case $c_1(X)>0$}: Indeed, $\omega \in 2\pi c_1(X).$ By Lemmas \ref{chi} and \ref{zch}, $-f^* \omega \in 2\pi c_1(X)$ is also a KE form. So Theorem \ref{uniq}.3 implies that $-f^* \omega=a^* \omega,$ for some $a \in \Aut_0(X).$ Since $f$ is involutive, $-(a \circ f)^* \omega =\omega,$ and then the bi-anti-holomorphic map $a \circ f$ is an isometry of $(X,g),$ by Lemma \ref{spg}. Let $X^{a \circ f}$ denote the fixed point set of $a \circ f.$ As in the previous $2$ cases, we conclude, via Theorem \ref{totig} and Lemma \ref{totrk}, that $(X^{a \circ f}, g|_{X^{a \circ f}})$ is totally geodesic and totally real. Then, Lemma \ref{iffKE} and Proposition \ref{mainp} imply that $g|_{X^{a \circ f}}$ is Einstein iff $\operatorname{tr}_{13}(\mathcal{R}^{g,J}_{X^{a \circ f}})=-C\operatorname{Id}_{\mathfrak{X}(X^{a \circ f})},$ for some $C\in \mathbb{R}.$
\end{proof}

In some cases, it should be possible to avoid using the uniqueness of KE metrics (c.f.\ Theorem \ref{uniq}) to conclude that the bi-anti-holomorphism at stake is an isometry. Indeed, let $(X,\omega)$ be a KE manifold, and $f$ be an anti-holomorphism of $X$ with fixed point set an $n$-dimensional submanifold $X^f.$ Let $\{(U_{\alpha}, \phi_{\alpha})\}_{\alpha \in A}$ be a holomorphic atlas such that for each $\alpha \in A,$ $\omega|_{U_{\alpha}}=i\partial \overline{\partial} \chi_{\alpha}$ and $f^* \chi_{\alpha}=\chi_{\alpha}.$ Then, we can apply Lemma \ref{L2} to deduce that $f$ is an isometry of $(U_{\alpha},g|_{U_{\alpha}}),$ hence an isometry of $(X,g).$ Certainly, the totally geodesic totally real submanifold $(X^f, g|_{X^f})$ is Einstein iff condition \ref{Einscondition} holds true. 

It would be interesting to find concrete examples of Einstein manifolds that are submanifolds of a compact KE manifold via Theorem \ref{maint}. However, doing so would require having a closed form expression of the curvature endomorphism of the metric on the ambient manifold.

\begin{remark}\label{FanoEcase}
If $X$ is Fano K{\"a}hler-Einstein, it is known that $\Aut(X)$ is a complex reductive group. More precisely, if $\omega$ is a K{\"a}hler-Einstein metric, and $\Isom(X,\omega)$ is the maximal compact subgroup of $\Aut(X)$ of holomorphic isometries of $\omega,$ then $\Aut(X)$ coincides with the complexification $\Isom(X,\omega)^{\mathbb{C}}$ \cite{Matsushima}. So in fact, the holomorphism $a$ in case $c_1(X)>0$ of Theorem \ref{maint} is an element of $\Isom_0(X,\omega)^{\mathbb{C}}.$ Thus, if it so happens that $a \in \Isom_0(X,\omega),$ the $3^{rd}$ case of Theorem \ref{maint} gives a necessary and sufficient condition for the real part $X^f,$ if non-empty, to be Einstein. 
\end{remark}

\begin{proposition}\label{sanya}
Let $n\geq 1,$ and $(X,g,J)$ be a K{\"a}hler manifold of real dimension $2n.$ Let $f:(X,g,J)\to (X,g,J)$ be an isometric bi-anti-holomorphic mapping with fixed point set $X^f.$ If $X^f$ is an $n$-dimensional submanifold, then it is Lagrangian. 
\end{proposition}

\begin{proof}
Let $\zeta,\eta \in \mathfrak{X}(X^f).$ Then,
\begin{equation*}
\begin{split}
\omega(\zeta,\eta)&=g(J\zeta,\eta)\\
&=(f^* g)(J\zeta,\eta)\\
&=g(f_* J\zeta,f_*\eta)\\
&=-g(Jf_* \zeta,f_* \eta)\\
&=-g(J\zeta,\eta)\\
&=-\omega(\zeta,\eta),
\end{split}
\end{equation*}
so $\omega(\zeta,\eta)=0.$
\end{proof}

In particular, if $f$ is an isometric real structure and the real part $X^f$ is non-empty, then it is Lagrangian. The proof of Theorem \ref{maint}, and the above proposition have the following consequence.

\begin{corollary}\label{Lagcoro}
Let $n\geq 1,$ and $(X,g,J)$ be a compact K{\"a}hler-Einstein manifold of real dimension $2n.$ Let $f$ be a bi-anti-holomorphism of $(X,J),$ and denote its set of fixed points by $X^f.$ Assume that $X^f$ is an $n$-dimensional submanifold. If $c_1(X) < 0,$ then $X^f$ is Lagrangian. If $c_1(X) = 0$ and $[-f^* \omega]=[\omega],$ then $X^f$ is Lagrangian. If $c_1(X)>0$ and $f$ is involutive, there exists a holomorphism $a \in Aut_0(X)$ such that $a\circ f$ is an isometry, and for any such $a,$ if the fixed point set $X^{a\circ f}$ is an $n$-dimensional submanifold, then it is also Lagrangian.
\end{corollary} 

\textbf{Acknowledgments.} I thank the referee and the editor for their helpful feedback. I thank Richard Hind, Jean-Pierre Bourguignon, Carlos Simpson, and Maxim Kontsevich for their insightful comments. The research leading to these results has received funding from the European Research Council (ERC) under the European Union's Ninth Framework Programme Horizon Europe (ERC Synergy Project Malinca, Grant Agreement n.\ 101167526).

\begin{minipage}[t]{10cm}
\begin{flushleft}
\textsc{G.\ Clemente}

Institut de recherche en informatique fondamentale

CNRS UMR 8243

Paris, 75013, France

e-mail: gabriella.clemente@cnrs.fr
\end{flushleft}
\end{minipage}

\begin{thebibliography}{99}
\frenchspacing \small

\bibitem{AL} {\sc F.R.\ Al-Solamy}, {\it Submanifolds of Einstein manifolds}, FJMS {\bf 28.3} (2008), 657-666.

\bibitem{Aubin} {\sc T.\ Aubin}, {\it {\'E}quations du type Monge-Amp{\`e}re sur les vari{\'e}t{\'e}s k{\"a}hleriennes compactes}, C.\ R.\ Acad.\ Sci.\ Paris S{\'e}r.\ A-B {\bf 283.3} Aiii (1976), A119 -- A121.

\bibitem{BM} {\sc S.\ Bando and T.\ Mabuchi}, {\it Uniqueness of Einstein K{\"a}hler metrics modulo connected group actions}, Algebraic geometry, Sendai, Adv.\ Stud.\ Pure Math. {\bf 10} (1985), 11-40.

\bibitem{CDS} {\sc X.X. Chen}, {\sc S. Donaldson} and {\sc S. Sun}, {\it K{\"a}hler-Einstein metrics on Fano manifolds: I -- III}, J. Amer. Math. Soc. {\bf 28.1} (2015), 183-278.

\bibitem{REAL} {\sc C. Ciliberto} and {\sc C. Pedrini}, {\it Real abelian varieties and real algebraic curves}, Lectures in real geometry, de Gruyter expositions in mathematics 23 (ed. F. Broglia), de Gruyter, Berlin, 1996. 

\bibitem{DJ} {\sc M.W. Davis} and {\sc T. Januszkiewicz}, {\it Convex polytopes, coxeter orbifolds and torus actions}, Duke Math. J. {\bf 62.2} (1991), 417-451.

\bibitem{cg} {\sc D. Huybrechts}, {\it Complex geometry: an introduction}, Universitext, Springer, Berlin, 2005.

\bibitem{Kir} {\sc W.P.A. Klingenberg}, {\it Riemannian geometry}, de Gruyter studies in mathematics 1, de Gruyter, Berlin, 1995.

\bibitem{Koiso} {\sc N.\ Koiso}, {\it Hypersurfaces of Einstein manifolds}, Ann.\ Sci.\ {\'E}c.\ Norm.\ Sup{\'e}r.\ $4^e$ s{\'e}rie, {\bf 14.4} (1981), 433 -- 443.

\bibitem{Matsushima} {\sc Y. Matsushima}, {\it Sur la structure du groupe d'hom{\'e}omorphismes analytiques d'une certaine vari{\'e}t{\'e} k{\"a}hl{\'e}rienne}, Nagoya Math.\ J. {\bf 11} (1957), 145 -- 150.

\bibitem{c1} {\sc J.W. Milnor} and {\sc J.D. Stasheff}, {\it Characteristic Classes}, Annals of mathematics studies 76, Princeton University Press, Princeton, 1974.

\bibitem{SF} {\sc F. Sottile}, {\it Toric ideals, real toric varieties, and the algebraic moment map}, Topics in algebraic geometry and geometric modeling, Contemp. Math. 334, American Mathematical Society, Providence, 2003. 

\bibitem{Tian} {\sc G. Tian}, {\it K-stability and K{\"a}hler Einstein metrics}, Commun. Pure Appl. Math. {\bf 68.7} (2015), 1085-1156.

\bibitem{Yano} {\sc K.\ Yano}, {\it Totally real submanifolds of a Kaehlerian manifold}, J.\ Differ.\ Geom.\ {\bf 11} (1976), 351 -- 359.

\bibitem{Yau} {\sc S-T. Yau}, {\it On the Ricci curvature of a compact K{\"a}hler manifold and the complex Monge-Amp{\`e}re equation}, I.\ Comm.\ Pure Appl.\ Math.\ {\bf 31} (1978), 339 -- 411.
\end{thebibliography}
\end{document}